\documentclass[11pt]{article}

\input diagram.tex
\usepackage{amsmath} 
\usepackage[mathscr]{eucal}
\usepackage{amssymb} 
\usepackage{theorem}
\usepackage{tikz-cd}
\usetikzlibrary{positioning}

\usepackage{enumerate}

\theoremstyle{change}  
\newtheorem{theorem}{Theorem}[section] 
\newtheorem{lemma}[theorem]{Lemma}  

\newtheorem{corollary}[theorem]{Corollary}
\theorembodyfont{\rmfamily}  
\newtheorem{Remark}[theorem]{Remark}

\newtheorem{definition}[theorem]{Definition}

\newenvironment{proof}{\noindent{\bf Proof}\ }{\qed\bigskip}

\newcommand{\Aut}{\mathrm{Aut}}

\newcommand{\calF}{\mathcal{F}}

\newcommand{\calI}{\mathcal{I}}

\newcommand{\calP}{\mathcal{P}}

\newcommand{\catfont}{\mathsf}

\newcommand{\FF}{\mathbb{F}}

\newcommand{\Inn}{\mathrm{Inn}}

\newcommand{\lexp}[2]{\setbox0=\hbox{$#2$} \setbox1=\vbox to \ht0{}\,\box1^{#1\,}\!#2}

\newcommand{\lMod}[1]{\llap{\phantom{|}}_{#1}\catfont{Mod}}

\newcommand{\Out}{\mathrm{Out}}

\newcommand{\qed}{\nobreak\hfill
                  \vbox{\hrule\hbox{\vrule\hbox to 5pt
                  {\vbox to 8pt{\vfil}\hfil}\vrule}\hrule}}

\newcommand{\ZZ}{\mathbb{Z}}

\newcommand{\DD}{D^{\Delta}}

\newcommand{\TD}{T^{\Delta}}

\newcommand{\FFTD}{\mathbb{F}T^{\Delta}}
\newcommand{\RTD}{RT^{\Delta}}

\newcommand{\Proj}{\mathrm{Proj}}
\newcommand{\Fppk}[1]{\calF_{#1pp_k}^\Delta}
\def\dom{\backslash}

\title{Stable functorial equivalence of blocks}
\author{Serge Bouc and Deniz Y\i lmaz}
\date{}
\providecommand{\keywords}[1]
{
  \small\smallskip\par	
  \hspace{2ex}\textbf{Keywords:} #1
}
\providecommand{\msc}[1]
{
  \small\smallskip\par	
  \hspace{2ex}\textbf{MSC2020:} #1%
}

\begin{document}
\sloppy

\maketitle

\begin{abstract}
Let $k$ be an algebraically closed field of characteristic $p>0$, let $R$ be a commutative ring and let $\FF$ be an algebraically closed field of characteristic $0$.  We introduce the category $\overline{\Fppk{R}}$ of stable diagonal $p$-permutation functors over $R$. We prove that the category  $\overline{\Fppk{\FF}}$ is semisimple and give a parametrization of its simple objects in terms of the simple diagonal $p$-permutation functors.

We also introduce the notion of a stable functorial equivalence over $R$ between blocks of finite groups. We prove that if $G$ is a finite group and if $b$ is a block idempotent of $kG$ with an abelian defect group $D$ and Frobenius inertial quotient $E$, then there exists a stable functorial equivalence over $\FF$ between the pairs $(G,b)$ and $(D\rtimes E,1)$.
\end{abstract}

\msc{16S34, 20C20, 20J15.}

\section{Introduction}

In past decades, various notions of equivalences between blocks of finite groups have been studied such as splendid Morita equivalence, splendid Rickard equivalence, $p$-permutation equivalence, isotypies and perfect isometries (\cite{Broue1990}, \cite{BoltjeXu2008}, \cite{BoltjePerepelitsky2020}). These equivalences are related to prominent conjectures in modular representation theory such as Brou{\'e}'s abelian defect group conjecture (Conjecture 9.7.6 in~\cite{linckelmann2018}), Puig's finiteness conjecture (Conjecture 6.4.2 in~\cite{linckelmann2018}) and Donovan's conjecture (Conjecture 6.1.9 in~\cite{linckelmann2018}).

Recently, in \cite{BoucYilmaz2022} we introduced another equivalence of blocks, namely {\em functorial equivalence}, using the notion of diagonal $p$-permutation functors: Let $k$ be an algebraically closed field of characteristic $p>0$, let $\FF$ be an algebraically closed field of characteristic $0$ and let $R$ be a commutative ring. We denote by $Rpp_k^\Delta$ the category whose objects are finite groups and for finite groups $G$ and $H$ whose morphisms from $H$ to $G$ are the Grothendieck group $R\TD(G,H)$ of {\em diagonal} $p$-permutation $(kG,kH)$-bimodules. An $R$-linear functor from $Rpp_k^\Delta$ to $\lMod{R}$ is called a {\em diagonal $p$-permutation functor}. To each pair $(G,b)$ of a finite group $G$ and a block idempotent $b$ of $kG$, we associate a canonical diagonal $p$-permutation functor over $R$, denoted by $RT^\Delta_{G,b}$. If $(H,c)$ is another such pair, we say that $(G,b)$ and $(H,c)$ are {\em functorially equivalent over $R$} if the functors $RT^\Delta_{G,b}$ and $RT^\Delta_{H,c}$ are isomorphic.

In \cite{BoucYilmaz2022} we proved that the category of diagonal $p$-permutation functors over $\FF$ is semisimple, parametrized simple functors and provided three equivalent descriptions of the decomposition of the functor $\FFTD_{G,b}$ in terms of the simple functors (\cite[Corollary~6.15 and Theorem~8.22]{BoucYilmaz2022}). We proved that the number of isomorphism classes of simple modules, the number of ordinary characters, and the defect groups are preserved under functorial equivalences over $\FF$ (\cite[Theorem~10.5]{BoucYilmaz2022}).  Moreover we proved that for a given finite $p$-group $D$, there are only finitely many pairs $(G,b)$, where $G$ is a finite group and $b$ is a block idempotent of $kG$, up to functorial equivalence over $\FF$ (\cite[Theorem~10.6]{BoucYilmaz2022}) and we provided a sufficient condition for two blocks to be functorially equivalent over $\FF$ in the situation of Brou{\'e}'s abelian defect group conjecture (\cite[Theorem~11.1]{BoucYilmaz2022}).

In this paper, we introduce the notion of {\em stable} diagonal $p$-permutation functors and {\em stable} functorial equivalences. We denote by $\overline{Rpp_k^\Delta}$ the quotient category of $Rpp_k^\Delta$ by the morphisms that factor through the trivial group. A {\em stable diagonal $p$-permutation functor over $R$} is an $R$-linear functor from $\overline{Rpp_k^\Delta}$ to $\lMod{R}$, or equivalently, a diagonal $p$-permutation functor which vanishes at the trivial group. In particular, the simple diagonal $p$-permutation functors $S_{L,u,V}$ with $L\neq 1$ are (simple) stable diagonal $p$-permutation functors. Our first main result is the following.

\begin{theorem}\label{thm semisimplicity}
The category $\overline{\Fppk{\FF}}$ of stable diagonal $p$-permutation functors over $\FF$ is semisimple. The simple stable diagonal $p$-permutation functors are precisely the simple diagonal $p$-permutation functors $S_{L,u,V}$ with $L\neq 1$. 
\end{theorem}

Given a finite group $G$ and a block idempotent $b$ of $kG$, we define a stable diagonal $p$-permutation functor $\overline{R\TD_{G,b}}$ similar to $R\TD_{G,b}$, see Definition~\ref{def stableblockfunctor}.  Note that $\overline{R\TD_{G,b}}$ is the zero functor if and only if $b$ has defect $0$. We say that two pairs $(G,b)$ and $(H,c)$ are {\em stably functorially equivalent over $R$} if the functors $\overline{R\TD_{G,b}}$ and $\overline{R\TD_{H,c}}$ are isomorphic. For a block algebra $kGb$, let $k(kGb)$ and $l(kGb)$ denote the number of irreducible ordinary characters and the number of irreducible Brauer characters of $b$, respectively.  

\begin{theorem}\label{thm stablefunequiv} Let $b$ be a block idempotent of $kG$ and let $c$ be a block idempotent of $kH$. 

\smallskip
{\rm (i)} The pairs $(G,b)$ and $(H,c)$ are stably functorially equivalent over $\FF$ if and only if the multiplicities of $S_{L,u,V}$ in $\FFTD_{G,b}$ and $\FFTD_{H,c}$ are the same for any simple diagonal $p$-permutation functor $S_{L,u,V}$ with $L\neq 1$. In this case, $(G,b)$ and $(H,c)$ are functorially equivalent over $\FF$ if and only if $l(kGb)=l(kHc)$.

\smallskip
{\rm (ii)} If the pairs $(G,b)$ and $(H,c)$ are stably functorially equivalent over $\FF$, then $b$ and $c$ have isomorphic defect groups and one has
\begin{align*}
k(kGb)-l(kGb)=k(kHc)-l(kHc)\,.
\end{align*}
\end{theorem}

We also consider the blocks with abelian defect groups and Frobenius inertial quotient.

\begin{theorem}\label{thm stableequivalence}
Let $G$ be a finite group, $b$ a block idempotent of $kG$ with a nontrivial abelian defect group $D$. Let $E=N_G(D,e_D)/C_G(D)$ denote the inertial quotient of $b$. Suppose that $E$ acts freely on $D\setminus \{1\}$. Then there exists a stable functorial equivalence over $\FF$ between $(G,b)$ and $(D\rtimes E,1)$. 
\end{theorem}

\begin{corollary}\label{cor p-permutationequivalence}
Assume the notation in Theorem~\ref{thm stableequivalence}.

\smallskip
{\rm (i)} There exists a functorial equivalence over $\FF$ between $(G,b)$ and $(D\rtimes E,1)$ if and only if $l(kGb)=l(k(D\rtimes E))$. 

\smallskip
{\rm (ii)} Suppose that $E$ is abelian.  Then there exists a functorial equivalence over $\FF$ between $(G,b)$ and $(D\rtimes E,1)$ if and only if $(G,b)$ and $(D\rtimes E,1)$ are $p$-permutation equivalent.
\end{corollary}

In Section~\ref{sec prelims} we recall diagonal $p$-permutation functors and functorial equivalences between blocks. In Section~\ref{sec stablefunctors} we introduce the category of stable diagonal $p$-permutation functors and prove Theorem~\ref{thm semisimplicity}. In Section~\ref{sec stableequivalences} we introduce the notion of stable functorial equivalences between blocks and prove Theorem~\ref{thm stablefunequiv}. Finally, in Section~\ref{sec frobeniusinertialquotient} we prove Theorem~\ref{thm stableequivalence} and Corollary~\ref{cor p-permutationequivalence}.

\section{Preliminaries}\label{sec prelims}
{\rm (a)} Let $(P,s)$ be a pair where $P$ is a $p$-group and $s$ is a generator of a $p'$-group acting on $P$. We write $P\langle s\rangle := P\rtimes \langle s\rangle$ for the corresponding semi-direct product. We say that two pairs $(P,s)$ and $(Q,t)$ are {\em isomorphic} and write $(P,s)\cong (Q,t)$, if there is a group isomorphism $f: P\langle s\rangle \to Q\langle t\rangle$ that sends $s$ to a conjugate of $t$.  We set $\Aut(P,s)$ to be the group of the automorphisms of the pair $(P,s)$ and $\Out(P,s)=\Aut(P,s)/\Inn(P\langle s\rangle)$. Recall from \cite{BoucYilmaz2020} that a pair $(P,s)$ is called a {\em $\DD$-pair}, if $C_{\langle s\rangle}(P)=1$.

\smallskip
{\rm (b)} Let $G$ and $H$ be finite groups. We denote by $T(G)$ the Grothendieck group of $p$-permutation $kG$-modules and by $\TD(G,H)$ the Grothendieck group of $p$-permutation $(kG,kH)$-bimodules whose indecomposable direct summands have twisted diagonal vertices. Let $Rpp_k^\Delta$ denote the following category:
\begin{itemize}
\item objects: finite groups.
\item $\mathrm{Mor}_{Rpp_k^\Delta}(G,H) = R\otimes_{\ZZ}T^{\Delta}(H,G)=RT^{\Delta}(H,G)$.
\item composition is induced from the tensor product of bimodules.
\item $\mathrm{Id}_{G}=[kG]$.
\end{itemize} 
An $R$-linear functor from $Rpp_k^\Delta$ to $\lMod{R}$ is called a \textit{diagonal $p$-permutation functor} over $R$. Together with natural transformations, diagonal $p$-permutation functors form an abelian category $\Fppk{R}$.

\smallskip
{\rm (c)} Recall from \cite{BoucYilmaz2022} that the category $\Fppk{\FF}$ is semisimple. Moreover, the simple diagonal $p$-permutation functors, up to isomorphism, are parametrized by the isomorphism classes of triples $(L,u,V)$ where $(L,u)$ is a $\DD$-pair, and $V$ is a simple $\FF\Out(L,u)$-module (see \cite[Sections~6~and~7]{BoucYilmaz2022} for more details on simple functors).

\smallskip
{\rm (d)} Let $G$ be a finite group and $b$ a block idempotent of $kG$.  Recall from \cite{BoucYilmaz2022} that the block diagonal $p$-permutation functor $R\TD_{G,b}$ is defined as
\begin{align*}
R\TD_{G,b}: Rpp_k^{\Delta}&\to \lMod{R}\\
H&\mapsto RT^\Delta(H,G)\otimes_{kG} kGb\,.
\end{align*}
See \cite[Section~8]{BoucYilmaz2022} for the decomposition of $\FFTD_{G,b}$ in terms of the simple functors $S_{L,u,V}$. 

\smallskip
{\rm (e)} Let $b$ be a block idempotent of $kG$ and let $c$ be a block idempotent of~$kH$. We say that the pairs $(G,b)$ and $(H,c)$ are {\em functorially equivalent} over $R$, if the corresponding diagonal $p$-permutation functors $R\TD_{G,b}$ and $R\TD_{H,c}$ are isomorphic in $\Fppk{R}$ (\cite[Definition~10.1]{BoucYilmaz2022}).  By \cite[Lemma~10.2]{BoucYilmaz2022} the pairs $(G,b)$ and $(H,c)$ are functorially equivalent over~$R$ if and only if there exists $\omega\in bR\TD(G,H)c$ and $\sigma\in cR\TD(H,G)b$ such that
\begin{align*}
\omega \cdot_G \sigma = [kGb] \quad \text{in} \quad b\RTD(G,G)b \quad \text{and} \quad \sigma \cdot_H \omega = [kHc] \quad \text{in} \quad c\RTD(H,H)c \,.
\end{align*}

\section{Stable diagonal $p$-permutation functors}\label{sec stablefunctors}

In this section we introduce the category of stable diagonal $p$-permutation functors. 

For a finite group $G$, let $P(G)$ denote the subgroup of $T(G)$ generated by the indecomposable projective $kG$-modules. Let also $\overline{T(G)}$ denote the quotient group $T(G)/P(G)$. For $X\in T(G)$, we denote by $\overline{X}$ the image of $X$ in $\overline{T(G)}$. If $H$ is another finite group, we define $P(G,H)$ and $\overline{\TD(G,H)}$ similarly. 

\begin{lemma}
For finite groups $G$ and $H$ one has $P(G,H)=T^\Delta(G,1)\circ T^\Delta(1,H)$ 
\end{lemma}
\begin{proof}
This follows from the fact that the projective indecomposable $k(G\times H)$-modules are of the form $P\otimes_k Q$ where $P$ and $Q$ are projective indecomposable $kG$ and $kH$-modules, respectively.
\end{proof}

\begin{definition}
Let $\overline{R pp_k^\Delta}$ denote the following category:
\begin{itemize}
\item objects: finite groups.
\item $\mathrm{Mor}_{\overline{Rpp_k^\Delta}}(G,H) = R\otimes_{\ZZ}\overline{T^{\Delta}(H,G)}=\overline{R\TD(H,G)}$.
\item composition is induced from the tensor product of bimodules.
\item $\mathrm{Id}_{G}=\overline{[kG]}$.
\end{itemize} 
\end{definition}

\begin{definition}
An $R$-linear functor $\overline{Rpp_k^\Delta}\to \lMod{R}$ is called a {\em stable diagonal $p$-permutation functor} over $R$.  Together with natural transformations, stable diagonal $p$-permutation functors form an abelian category $\overline{\Fppk{R}}$.
\end{definition}

\begin{Remark}
The functor
\begin{align*}
\Gamma: \overline{\Fppk{R}}\to \Fppk{R}
\end{align*}
obtained by composition with the projection $Rpp_k^\Delta\to \overline{Rpp_k^\Delta}$ gives a description of $\overline{\Fppk{R}}$ as a full subcategory of $\Fppk{R}$. Moreover, $\Gamma$ has a left adjoint $\Sigma$, constructed as follows: If $F$ is a diagonal $p$-permutation functor over $R$ and $G$ is a finite group, set
\begin{align*}
\overline{F}(G):=F(G)/\RTD(G,\textbf{1})F(\textbf{1})\,.
\end{align*}
Then $\overline{F}$ is a diagonal $p$-permutation functor, equal to the quotient of $F$ by the subfunctor generated by $F(\textbf{1})$. Obviously, $\overline{F}$ vanishes at the trivial group, so it is a stable diagonal $p$-permutation functor. The functor $\Sigma: F\mapsto \overline{F}$ is a left adjoint to the above functor $\Gamma$. In particular, $\overline{\Fppk{R}}$ is a reflective subcategory of $\Fppk{R}$.
\end{Remark}

Let $G$ be a finite group. Recall that by \cite[Corollary~8.23(i)]{BoucYilmaz2022}, the multiplicity of the simple diagonal $p$-permutation functor $S_{1,1,\FF}$ in the representable functor $\FFTD(-,G)$ is equal to the number $l(kG)$ of the isomorphism classes of simple $kG$-modules.  Let $\calI(-,G)$ denote the sum of simple subfunctors of $\FFTD(-,G)$ isomorphic to $S_{1,1,\FF}$. Let also $\FF\Proj(-,G)$ denote the subfunctor of $\FFTD(-,G)$ sending a finite group $H$ to $\FF\Proj(H,G)$.  

\begin{lemma}\label{lem projectivesubfunctor}
The subfunctors $\calI(-,G)$ and $\FF\Proj(-,G)$ of the representable functor $\FFTD(-,G)$ are isomorphic.
\end{lemma}
\begin{proof}
For finite groups $G$ and $H$, the number of isomorphism classs of projective indecomposable $k(G\times H)$-modules, or equivalently, the number of isomorphism classes of simple $k(G\times H)$-modules is equal to the number of conjugacy classes of $p'$-elements of $G\times H$. Hence the $\FF$-dimension of the evaluation $\FF\Proj(H,G)$ is equal to
\begin{align*}
l(k(G\times H))=l(kG)l(kH)
\end{align*}
which is equal to the $\FF$-dimension of $l(kG)S_{1,1,\FF}(H)$, and hence to the $\FF$-dimension of $\calI(H,G)$. 

Note that $\FF\Proj(-,G)$ is isomorphic to the functor
\begin{align*}
\FFTD(-,1)\circ\FFTD(1,G)\,.
\end{align*}
Moreover $S_{L,u,V}(1)=0$ for $L\neq 1$, and hence $\FFTD(1,G)=\calI(1,G)$. Therefore, one has
\begin{align*}
\FF\Proj(-,G)\cong \FFTD(-,1)\circ\FFTD(1,G) = \FFTD(-,1)\circ\calI(1,G) \subseteq \calI(-,G)\,.
\end{align*}
Since the $\FF$-dimensions of $\FF\Proj(H,G)$ and $\calI(H,G)$ are the same for any finite group $H$, it follows that $\FF\Proj(-,G)\cong \calI(-,G)$. 
\end{proof}

{\em Proof of Theorem~\ref{thm semisimplicity}:}
For a finite group $G$, the representable diagonal $p$-permutation functor $\FFTD(-,G)$ decomposes as a direct sum of simple functors $S_{L,u,V}$, and hence we have
\begin{align*}
\FFTD(-,G)\cong \calI(-,G)\bigoplus_{\substack{(L,u,V)\\ L\neq 1}} S_{L,u,V}^{m_{L,u,V}}\,,
\end{align*}
for some nonnegative integers $m_{L,u,V}$, where $(L,u,V)$ runs over a set of isomorphism classes of $\DD$-pairs $(L,u)$ with $L\neq 1$, and simple $\FF\Out(L,u)$-modules $V$.  By Lemma~\ref{lem projectivesubfunctor}, the representable stable diagonal $p$-permutation functor $\overline{\FFTD(-,G)}$ is isomorphic to the direct sum 
\begin{align*}
\bigoplus_{\substack{(L,u,V)\\ L\neq 1}} S_{L,u,V}^{m_{L,u,V}}
\end{align*}
of simple diagonal $p$-permutation functors,
and each of these simple functors is a simple stable diagonal $p$-permutation functor. Since the functor category $\overline{\Fppk{\FF}}$ is generated by the representable functors the result follows. 
\qed

\section{Stable functorial equivalences}\label{sec stableequivalences}
Let $G$ and $H$ be finite groups.

\begin{definition}\label{def stableblockfunctor}
Let $b$ a block idempotent of $kG$.  The stable diagonal $p$-permutation functor $\overline{R\TD_{G,b}}$ is defined as
\begin{align*}
\overline{R\TD_{G,b}}: \overline{Rpp_k^{\Delta}}&\to \lMod{R}\\
H&\mapsto \overline{RT^\Delta(H,G)\otimes_{kG} kGb}\,.
\end{align*}
\end{definition}

See Section~\ref{sec prelims}{\rm (d)} for the definition of $R\TD_{G,b}$ and note that $\overline{R\TD_{G,b}}=\Sigma\left(R\TD_{G,b}\right)$.

\begin{definition}
Let $b$ be a block idempotent of $kG$ and let $c$ be a block idempotent of $kH$. We say that the pairs $(G,b)$ and $(H,c)$ are {\em stably functorially equivalent} over $R$, if their corresponding stable diagonal $p$-permutation functors $\overline{R\TD_{G,b}}$ and $\overline{R\TD_{H,c}}$ are isomorphic in $\overline{\Fppk{R}}$.
\end{definition}

\begin{lemma}
Let $b$ be a block idempotent of $kG$ and let $c$ be a block idempotent of $kH$. 

\smallskip
{\rm (a)} The pairs $(G,b)$ and $(H,c)$ are stably functorially equivalent over $R$ if and only if there exists $\omega\in bR\TD(G,H)c$ and $\sigma\in cR\TD(H,G)b$ such that
\begin{align*}
\omega \cdot_G \sigma = [kGb]+[P] \quad \text{in} \quad b\RTD(G,G)b \quad \text{and} \quad \sigma \cdot_H \omega = [kHc]+[Q] \quad \text{in} \quad c\RTD(H,H)c 
\end{align*}
for some $P\in R\Proj(kGb,kGb)$ and $Q\in R\Proj(kHc,kHc)$. 

\smallskip
{\rm (b)} If the pairs $(G,b)$ and $(H,c)$ are functorially equivalent over $R$, then they are also stably functorially equivalent over $R$.
\end{lemma}
\begin{proof}
By the Yoneda lemma, the $(G,b)$ and $(H,c)$ are stably functorially equivalent over $R$ if and only if there exists $\overline{\omega}\in \overline{bR\TD(G,H)c}$ and $\overline{\sigma}\in \overline{cR\TD(H,G)b}$ such that
\begin{align*}
\overline{\omega} \cdot_G \overline{\sigma} = \overline{[kGb]} \quad \text{in} \quad \overline{b\RTD(G,G)b} \quad \text{and} \quad \overline{\sigma} \cdot_H \overline{\omega} = \overline{[kHc]} \quad \text{in} \quad \overline{c\RTD(H,H)c} \,.
\end{align*}
Hence {\rm (a)} follows and {\rm (b)} is clear. 
\end{proof}

{\em Proof of Theorem~\ref{thm stablefunequiv}:}
{\rm (i)} The first statement follows from Theorem~\ref{thm semisimplicity} and the second statement follows since the multiplicity of the simple functor $S_{1,1,\FF}$ in $\FFTD_{G,b}$ is equal to $l(kGb)$.

\smallskip
{\rm (ii)} The first statement follows from the proof of \cite[Theorem~10.5(iii)]{BoucYilmaz2022} and the second statement follows from the proof of  \cite[Theorem~10.5(ii)]{BoucYilmaz2022}. 
\qed

\section{Blocks with Frobenius inertial quotient}\label{sec frobeniusinertialquotient}

\smallskip
{\rm (a)} Let $G$ be a finite group, $b$ a block idempotent of $kG$ with a nontrivial abelian defect group $D$. Let $(D,e_D)$ be a maximal $b$-Brauer pair and let $E=N_G(D,e_D)/C_G(D)$ denote the inertial quotient of $b$. Suppose that $E$ acts freely on $D\setminus \{1\}$.

The condition that the action of $E$ on $D\setminus \{1\}$ is free is equivalent to requiring that $D\rtimes E$ is a Frobenius group. Let $\calF_b$ be the fusion system of $b$ with respect to $(D,e_D)$. Then $\calF_b$ is equal to the fusion system $\calF_{D\rtimes E}(D)$ on $D$ determined by $D\rtimes E$.

\smallskip
{\rm (b)} Let $S_{L,u,V}$ be a simple diagonal $p$-permutation functor such that $L$ is nontrivial and isomorphic to a subgroup of $D$. Recall that by \cite[Theorem~8.22]{BoucYilmaz2022} the multiplicity of $S_{L,u,V}$ in $\FFTD_{G,b}$ is equal to the $\FF$-dimension of
\begin{align*}
\bigoplus_{(P,e_P) \in [\calF_b]} \bigoplus_{\pi \in [N_G(P,e_P)\dom\calP_{(P,e_P)}(L,u)/\Aut(L,u)]} \FF\Proj(ke_PC_G(P),u)\otimes_{\Aut(L,u)_{\overline{(P,e_P,\pi)}}} V\,,
\end{align*}
where $[\calF_b]$ denotes a set of isomorphism classes of objects in $\calF_b$,  $\calP_{(P,e_P)}$ is the set of group isomorphisms $\pi:L\to P$ with $\pi i_u \pi^{-1}\in \Aut_{\calF_b}(P,e_P)$, and $\Aut(L,u)_{\overline{(P,e_P,\pi)}}$ is the stabilizer in $\Aut(L,u)$ of the $G$-orbit of $(P,e_P,\pi)$.  Since $b$ is a block with Frobenius inertial quotient, the block $kC_G(P)e_P$ is nilpotent for every nontrival subgroup $P$ of $D$, see for instance \cite[Theorem~10.5.2]{linckelmann2018}. Therefore, we have $l(ke_PC_G(P))=1$, and hence the multiplicity formula reduces to
\begin{align*}
\bigoplus_{(P,e_P) \in [\calF_b]} \bigoplus_{\pi \in [N_G(P,e_P)\dom\calP_{(P,e_P)}(L,u)/\Aut(L,u)]} V^{\Out(L,u)_{\overline{(P,e_P,\pi)}}}\,.
\end{align*}

Let $\mathcal{Q}_{D\rtimes E,p}$ denote the set of pairs $(P,s)$ of $p$-subgroups $P$ of $D\rtimes E$ and $p'$-elements $s$ of $N_{D\rtimes E}(P)$. Let also $[\mathcal{Q}_{D\rtimes E,p}]$ denote a set of representatives of $D\rtimes E$-orbits on $\mathcal{Q}_{D\rtimes E,p}$ under the conjugation map. Recall from \cite[Corollary~7.4]{BoucYilmaz2022} that the multiplicity of $S_{L,u,V}$ in $\FFTD_{D\rtimes E}$ is equal to the $\FF$-dimension of
\begin{align*}
\bigoplus_{\substack{(P,s)\in[\mathcal{Q}_{D\rtimes E,p}]\\ (\tilde{P},\tilde{s})\cong (L,u)}} V^{N_{D\rtimes E}(P,s)}\,,
\end{align*}
where for a pair $(P,s)\in \mathcal{Q}_{D\rtimes E,p}$ with $(\tilde{P},\tilde{s})\cong (L,u)$,  we fix an isomorphism $\phi_{P,s}:L\to P$ with $\phi_{P,s}(\lexp{u}l)=\lexp{s}\phi_{P,s}(l)$ for all $l\in L$ and we view $V$ as an $\FF N_{D\rtimes E}(P,s)$-module via the group homomorphism
\begin{align}
N_G(P,s)\to \Out(L,u)
\end{align}
that sends $g\in N_G(P,s)$ to the image of $\phi_{P,s}^{-1}\circ i_g\circ \phi_{P,s}$ in $ \Out(L,u)$.

\smallskip
{\rm (c)} Let $\calP_b(G,L,u)$ denote the set of triples $(P,e,\pi)$ where $(P,e)\in \calF_b$ and $\pi\in \calP_{(P,e_P)}(L,u)$.  Let also $\mathcal{Q}_{D\rtimes E,p}(L,u)$ denote the set of pairs $(P,s)$ in $\mathcal{Q}_{D\rtimes E,p}$ with the property that $(\tilde{P},\tilde{s})\cong (L,u)$.

If $(P,e,\pi)\in \calP_b(G,L,u)$, then $\pi i_u \pi^{-1}\in \Aut_{\calF_b}(P,e_P)$ by definition and since $\calF_b$ is equal to $\calF_{D\rtimes E}(D)$, it follows that there exists a $p'$-element $s$ of $N_{D\rtimes E}(P)$ with $\pi i_u \pi^{-1}=i_s$. This implies by \cite[Lemma~3.3]{BoucYilmaz2022} that $(\tilde{P},\tilde{s})\cong (L,u)$ and therefore we have a map
\begin{align*}
\Psi: \calP_b(G,L,u)\to \mathcal{Q}_{D\rtimes E,p}(L,u), \quad (P,e,\pi)\mapsto (P,s)\,.
\end{align*}

\begin{lemma}\label{lem bijection}
The map $\Psi$ induces a bijection
\begin{align*}
\overline{\Psi}:[G\dom \calP_b(G,L,u)/\Aut(L,u)]\to [\mathcal{Q}_{D\rtimes E,p}(L,u)]\,.
\end{align*}
\end{lemma}
\begin{proof}
First we show that the map $\overline{\Psi}$ is well-defined. Let $(P,e,\pi)$ and $(Q,f,\rho)$ be two elements in $\calP_b(G,L,u)$ that lie in the same $G\times \Aut(L,u)$-orbit. We need to show that $\overline{\Psi}(P,e,\pi)=\overline{\Psi}(Q,f,\rho)$. Write $\Psi(P,e,\pi)=(P,s)$ and $\Psi(Q,f,\rho)=(Q,t)$. Let $g\in G$ and $\varphi\in \Aut(L,u)$ such that
\begin{align*}
g\cdot (P,e,\pi)\cdot \varphi=(Q,f,\rho)\,.
\end{align*}
Then $(P,e)$ and $(Q,f)$ lie in the same isomorphism class in $[\calF_b]$ and hence $P$ and $Q$ are $D\rtimes E$-conjugate since $\calF_b=\calF_{D\rtimes E}(D)$. Thus, there exists $h\in D\rtimes E$ with $i_g=i_h:P\to Q$.  Hence $\rho=i_g\pi\varphi=i_h\pi\varphi:L\to Q$. Since $\varphi\in\Aut(L,u)$, one has $\varphi\circ i_u=i_u\circ\varphi$. Therefore,
\begin{align*}
i_t=\rho i_u\rho^{-1}=i_h\pi\varphi i_u \varphi^{-1}\pi^{-1}i_{h^{-1}}=i_h\pi i_u \pi^{-1}i_{h^{-1}}=i_hi_s i_{h^{-1}}=i_{hsh^{-1}}\,.
\end{align*}
This shows that $(Q,t)=h\cdot (P,s)$ and hence the map $\overline{\Psi}$ is well-defined.

Now we show that $\overline{\Psi}$ is surjective.  Let $(P,s)\in \mathcal{Q}_{D\rtimes E,p}(L,u)$. Since $(\tilde{P},\tilde{s})\cong (L,u)$, again by \cite[Lemma~3.3]{BoucYilmaz2022}, there exists $\pi:L\to P$ such that $\pi i_u=i_s\pi$, i.e. $\pi i_u\pi^{-1}=i_s:P\to P$. Since $\calF_{D\rtimes E}(D)=\calF_b$, it follows that there exists $g\in N_G(P,e)$ with $i_s=i_g$, and hence $(P,e,\pi)\in \calP_b(G,L,u)$ with $\overline{\Psi}(P,e,\pi)=(P,s)$. Thus, $\overline{\Psi}$ is surjective.

Finally, we show that $\overline{\Psi}$ is injective. Let $(P,e,\pi), (Q,f,\rho)\in \calP_b(G,L,u)$ be elements with $\overline{\Psi}(P,e,\pi)=\overline{\Psi}(Q,f,\rho)$. Write $(P,s)=\Psi(P,e,\pi)$ and $(Q,f)=\Psi(Q,f,\rho)$. Then there exists $h\in D\rtimes E$ such that
\begin{align*}
h\cdot (P,s)=(Q,t)\,.
\end{align*}
Again, there exists $g\in G$ such that $i_g=i_h:P\to Q$.  Define
\begin{align*}
\varphi:=\pi^{-1}\circ i_g^{-1}\circ \rho: L\to L\,.
\end{align*}
One has
\begin{align*}
\varphi\circ i_u&=\pi^{-1}\circ i_g^{-1}\circ \rho\circ i_u=\pi^{-1}\circ i_g^{-1}\circ i_t\circ \rho=\pi^{-1}\circ i_g^{-1}\circ i_g\circ i_s\circ i_g^{-1} \circ \rho \\&
= \pi^{-1}\circ i_s\circ i_g^{-1} \circ \rho = i_u\circ \pi^{-1}\circ i_g^{-1} \circ \rho =i_u\circ \varphi
\end{align*}
which shows that $\varphi\in \Aut(L,u)$. Moreover, one has
\begin{align*}
g\cdot (P,e,\pi)\cdot \varphi =(Q,f,\rho)
\end{align*}
and so the map $\overline{\Psi}$ is injective.
\end{proof}

\begin{lemma}\label{lem stabilizers}
Let $(P,e,\pi)\in [G\dom \calP_b(G,L,u)/\Aut(L,u)]$ and $(P,s)=\overline{\Psi}(P,e,\pi)\in [\mathcal{Q}_{D\rtimes E,p}]$.  Then the image of $N_{D\rtimes E}(P,s)$ in $\Out(L,u)$ is equal to $\Out(L,u)_{\overline{(P,e,\pi)}}$.
\end{lemma}

\begin{proof}
We have $\pi i_u\pi^{-1}=i_s$ and hence the image of $N_{D\rtimes E}(P,s)$ is given by
\begin{align*}
N_{D\rtimes E}(P,s)&\to \Out(L,u)\\
h &\mapsto \overline{\pi^{-1}\circ i_h\circ \pi}
\end{align*}
Note that since $\lexp{h}s=s$, we have $i_hi_s=i_si_h$, i.e., $i_h\pi i_u\pi^{-1}=\pi i_u\pi^{-1}i_h$. Therefore the image is 
\begin{align*}
\{\overline{\pi^{-1} i_h\pi} \mid h\in D\rtimes E, i_h:P\to P, \lexp{h}s=s\}&=\{\overline{\pi^{-1} i_g\pi} \mid g\in N_G(P,e), i_g\pi i_u\pi^{-1}=\pi i_u\pi^{-1}i_g\}\\&
=\{\overline{\pi^{-1} i_g\pi}\in\Out(L,u) \mid \pi^{-1} i_g\pi=i_g,  g\in N_G(P,e)\}\\&
=\Out(L,u)_{\overline{(P,e,\pi)}}
\end{align*}
as was to be shown.
\end{proof}

{\em Proof of Theorem~\ref{thm stableequivalence}:}
We need to show that for any $L\neq 1$, the multiplicities of a simple diagonal $p$-permutation functor $S_{L,u,V}$ in $\FFTD_{G,b}$ and in $\FFTD_{D\rtimes E}$ are equal.  But this follows from Lemma~\ref{lem bijection} and Lemma~\ref{lem stabilizers}.
\qed

{\em Proof of Corollary~\ref{cor p-permutationequivalence}:}
Part {\rm (i)} follows from Theorem~\ref{thm stablefunequiv}{\rm (i)} and Theorem~\ref{thm stableequivalence}. Part {\rm (ii)} follows from \cite[Theorem~10.5.10]{linckelmann2018}.
\qed

\centerline{\rule{5ex}{.1ex}}
\begin{flushleft}
Serge Bouc, CNRS-LAMFA, Universit\'e de Picardie, 33 rue St Leu, 80039, Amiens, France.\\
{\tt serge.bouc@u-picardie.fr}\vspace{1ex}\\
Deniz Y\i lmaz, Department of Mathematics, Bilkent University, 06800 Ankara, Turkey.\\
{\tt d.yilmaz@bilkent.edu.tr}
\end{flushleft}


\begin{thebibliography}{00}








 

        
        
    

\bibitem[Br90]{Broue1990}{\sc M.~Brou{\'e}:}
	Isom{\'e}tries parfaites, types de blocs, cat{\'e}gories d{\'e}riv{\'e}es. 
	{\sl Ast{\'e}risque} No. {\bf 181-182} (1990), 61--92.



\bibitem[BP20]{BoltjePerepelitsky2020}{\sc R.~Boltje, P.~Perepelitsky:}
        $p$-permutation equivalences between blocks of group algebras.
        arXiv:2007.09253.

\bibitem[BX08]{BoltjeXu2008}{\sc R.~Boltje, B.~Xu:}
	On $p$-permutation equivalences: between Rickard equivalences and
   isotypies
	{\sl Trans. Amer. Math. Soc.} {\bf 360}(10) (2008) 5067--5087.

\bibitem[BY20]{BoucYilmaz2020} {\sc S.~Bouc, D.~Y{\i}lmaz:}
	Diagonal $p$-permutation functors.
	{\sl J. Algebra} {\bf 556} (2020), 1036--1056.
	
\bibitem[BY22]{BoucYilmaz2022}{\sc S.~Bouc, D.~Y\i lmaz:}
	Diagonal $p$-permutation functors, semisimplicity, and functorial equivalence of blocks.
	{\sl Adv.  Math.} {\bf 411} (2022), 108799.


	
\bibitem[L18]{linckelmann2018}{\sc M.~Linckelmann:}
	The block theory of finite group algebras. {V}ol. {II}.
	Cambridge University Press, Cambridge, 2018.






\end{thebibliography}
\end{document}